\documentclass{cimart}

\usepackage{paracol}
\usepackage[all,cmtip]{xy}

\newcommand\Stone[1]{\fbox{\makebox[11mm]{\strut#1}}\kern2pt}

\title{Simple subquotients of relation modules}

\authors{Gustavo Costa, Lucas Queiroz Pinto, and Luis Enrique Ramirez}

\authorinfo[G. Costa]{Federal University of  ABC,  Brazil}{costa.pereira@ufabc.edu.br}

\authorinfo[L. Q. Pinto]{University of   São Paulo,  Brazil}{lucasqueirozp@ime.usp.br}

\authorinfo[L. E. Ramirez]{Federal University of  ABC,   Brazil}{luis.enrique@ufabc.edu.br}

\abstract{%
    In this paper, we provide an explicit tableau realization for all simple subquotients of any relation Gelfand--Tsetlin $\mathfrak{gl}(n)$-module.
    }

\keywords{
    Gelfand--Tsetlin modules, Gelfand--Tsetlin basis, Tableau realization.
    }

\acknowledgments{G.C. was partially supported by the Coordena\c{c}\~{a}o de Aperfei\c{c}oamento de Pessoal de N\'ivel Superior -- Brasil (CAPES) -- Finance Code 001; L.Q.P. was supported by CNPq grant 140582/2021-5; L.E.R. was
partially supported by CNPq grant 316163/2021-0, and Fapesp grant 2022/08948-2. The authors are grateful to the referees for their insightful comments and suggestions, which significantly improved the clarity and presentation of this paper.}

\msc{17B10 (primary).}

\VOLUME{34}
\YEAR{2026}
\ISSUE{2}
\NUMBER{8}
\DOI{https://doi.org/2507.01221/cm.15984}
\editinfo{July 3, 2025}{February 9, 2026}{Tiago Macedo, Jaqueline Mesquita, Rafael Potrie, and Mariel Saez}

\begin{document}

\section{Introduction}
 In the context of the representation theory of the Lie algebra $\mathfrak{gl}(n)$ of all $n\times n$ matrices over $\mathbb{C}$, a remarkable construction, due to I. Gelfand and M. Tsetlin \cite{GT50}, provides a presentation for every simple finite-dimensional module in terms of certain combinatorial objects called Gelfand--Tsetlin tableaux, and the explicit action of the generators of the algebra -- the Gelfand--Tsetlin formulas. In \cite{Zhe73}, it is shown that the Gelfand--Tsetlin basis is an eigenbasis for the action of a certain maximal commutative subalgebra $\Gamma$ of $U(\mathfrak{gl}(n))$ -- the Gelfand--Tsetlin subalgebra. 
 
 In \cite{DFO94}, the authors introduced and studied modules that can be decomposed into a direct sum of generalized eigenspaces with respect to the action of $\Gamma$ -- the Gelfand--Tsetlin modules. Moreover, the class of generic Gelfand--Tsetlin modules was constructed. These infinite-dimensional modules have a tableau basis like the simple finite-dimensional modules, and the action of $U(\mathfrak{gl}(n))$ is also given by the Gelfand--Tsetlin formulas.

Relation modules (see \cite{FRZ19}) depend on a directed graph and a certain collection of tableaux related to the graph, and were constructed with the aim of unifying several known constructions of Gelfand--Tsetlin modules (see \cite{GT50, LP79, DFO94, Maz98, Maz03}, among others). The relation modules approach is based on the construction of explicit bases subject to certain restrictions on the entries of Gelfand--Tsetlin tableaux, which prevent the singularities that arise in the Gelfand--Tsetlin formulas (formulas used to realize simple finite-dimensional $\mathfrak{gl}(n)$-modules \cite{GT50}). A different approach was proposed in \cite{FGR16}, where singularities are handled by generalizing the Gelfand--Tsetlin formulas themselves. Since then, several purely algebraic works addressing this problem have appeared (see, for example, \cite{Zad17, FGR17a, FGR17b, RZ18, EMV20}). Geometric methods have also been employed to study singular Gelfand--Tsetlin modules \cite{Vis18}. More recently, a classification of simple Gelfand--Tsetlin modules has been obtained by establishing a connection between principal Galois orders and Coulomb branches \cite{K+19, SW24, Web24}.
 In this paper, we provide an explicit basis for any simple subquotient of a relation module, generalizing the results obtained in \cite{FGR15} for generic modules. 

  This paper is divided as follows: in Section~\ref{sec:GTmod}, we define the Gelfand--Tsetlin modules and recall the construction of all simple finite-dimensional $\mathfrak{gl}(n)$-modules. In Section \ref{sec:RelGTmodules}, we consider relation modules and recall their main properties. In Section~\ref{sec:simplesubquotients}, we establish the main results of the paper. Finally, Section~\ref{sec:examples} is devoted to presenting examples illustrating the main results discussed in Section~\ref{sec:simplesubquotients}.

\section{Gelfand--Tsetlin modules}\label{sec:GTmod}

     In this section, we recall the definitions and main properties of a full subcategory of the category of $\mathfrak{gl}(n)$-modules, the so-called Gelfand--Tsetlin modules. Let us fix $n\geq 2$. In order to define Gelfand--Tsetlin modules, we start by constructing a maximal commutative subalgebra of $U(\mathfrak{gl}(n))$. For $m\leqslant n$, denote by $\mathfrak{gl}_m$ the Lie subalgebra
    of $\mathfrak{gl}(n)$ spanned by $\{ E_{ij}\,|\, i,j=1,\ldots,m \}$, and $U_m:=U({ \mathfrak{gl}_m})$. The strategy for constructing the Gelfand--Tsetlin subalgebra relies on the chain of inclusions
    \[
    \mathfrak{gl}_1\subset \mathfrak{gl}_2\subset \cdots \mathfrak{gl}_{n-1}\subset \mathfrak{gl}_n.
    \]
   
    \begin{definition} Set $U:=U_n$, and let $Z_{m}$ denote the center of $U_{m}$.
    \begin{itemize} 
    \item[(i)] The {\it standard Gelfand--Tsetlin subalgebra} $\Gamma$ 
    of $U$ is the subalgebra generated by  $
    \cup_{{i=1}}^{n}Z_i$.
    \item[(ii)] A \emph{Gelfand--Tsetlin module} is a $U$-module $M$ such that $M = 
    \oplus_{\chi\in\Gamma^*} M(\chi)$, where 
    \[M(\chi)=\{v\in M:\forall g\in\Gamma,\text{ }\exists k\in\mathbb{N} \text{ such that } (g-\chi(g))^{k}v=0 \}.\]    
    The support of $M$ is the set $\textup{supp}M:=\{\chi\in\Gamma^{*}:\, M(\chi)\ne 0\}.$
    \end{itemize}
    \end{definition}
     
 The realization of simple finite-dimensional $\mathfrak{gl}(n)$-modules given in \cite{GT50} is the main inspiration for the construction of relation modules. Moreover, this realization provides an eigenbasis for the action of $\Gamma$  \cite{Zhe73}. Namely, each basis vector is uniquely identified by the tuple of eigenvalues it produces when acted upon by the generators of $\Gamma$. These eigenvalues are traditionally arranged in a triangular array, known as a Gelfand--Tsetlin tableau, which respects the branching rules of the chain of subalgebras. Let us recall the result.

\begin{definition}
A configuration of complex numbers 
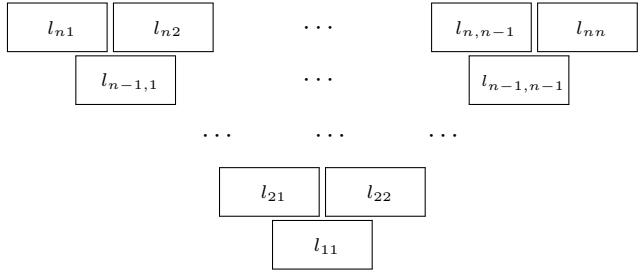
\begin{center}
\Stone{\mbox{ \scriptsize {$l_{n1}$}}}\Stone{\mbox{ \scriptsize {$l_{n2}$}}}\hspace{1cm} $\cdots$ \hspace{1cm} \Stone{\mbox{ \scriptsize {$l_{n,n-1}$}}}\Stone{\mbox{ \scriptsize {$l_{nn}$}}}\\[0.2pt]
\Stone{\mbox{ \scriptsize {$l_{n-1,1}$}}}\hspace{1.5cm} $\cdots$ \hspace{1.5cm} \Stone{\mbox{ \tiny {$l_{n-1,n-1}$}}}\\[0.3cm]
\hspace{0.2cm}$\cdots$ \hspace{0.8cm} $\cdots$ \hspace{0.8cm} $\cdots$\\[0.3cm]
\Stone{\mbox{ \scriptsize {$l_{21}$}}}\Stone{\mbox{ \scriptsize {$l_{22}$}}}\\[0.2pt]
\Stone{\mbox{ \scriptsize {$l_{11}$}}}
\end{center}
is called a \emph{Gelfand--Tsetlin tableau}. A Gelfand--Tsetlin tableau is called \emph{standard} if its entries satisfy
\[
l_{ki}-l_{k-1,i}\in\mathbb{Z}_{\geq 0}
\quad\text{and}\quad
l_{k-1,i}-l_{k,i+1}\in\mathbb{Z}_{>0}
\]
for all possible pairs of indices.
\end{definition}

\begin{theorem}[Gelfand--Tsetlin-1950]
 Let $\lambda=(\lambda_{1},\ldots,\lambda_{n})$ be an integral dominant $\mathfrak{gl}(n)$-weight (i.e., 
 $\lambda_i-\lambda_{i+1}\in\mathbb{Z}_{\ge 0}$, for all $1\leq i\leq n-1$).
 The simple finite-dimensional module $L(\lambda)$ has a realization of tableaux, where the vector space consists of all standard tableaux $T(L)$ with top row $l_{nj}=\lambda_j+j-1$, and the $\mathfrak{gl}(n)$-module structure is given by the Gelfand--Tsetlin formulas:
 \begin{equation}\label{eq:GTformulas}
\begin{split}
E_{k,k+1}(T(L))&=-\sum_{i=1}^{k}\left(\frac{\prod_{j=1}^{k+1}(l_{ki}-l_{k+1,j})}{\prod_{j\neq i}^{k}(l_{ki}-l_{kj})}\right)T(L+\delta^{ki}),\\ E_{k+1,k}(T(L))&=\sum_{i=1}^{k}\left(\frac{\prod_{j=1}^{k-1}(l_{ki}-l_{k-1,j})}{\prod_{j\neq i}^{k}(l_{ki}-l_{kj})}\right)T(L-\delta^{ki}),\\
E_{kk}(T(L))&=\left(k-1+\sum_{i=1}^{k}l_{ki}-\sum_{i=1}^{k-1}l_{k-1,i}\right)T(L),
\end{split}
\end{equation} where  $T(L\pm\delta^{ki})$ is the tableau obtained from $T(L)$ by adding $\pm 1$ to the $(k, i)$'s position of $T(L)$ (if the new tableau is not standard, then the result of the action is zero). 
\end{theorem}

 It is important to note that the denominators in the formulas \eqref{eq:GTformulas} involve differences of tableau entries. In the finite-dimensional case, the standardness conditions ensure these denominators never vanish. However, for the more general classes of modules we wish to study, where the standard constraints are relaxed, these denominators may become zero, leading to singularities. The relation modules introduced in the next section provide a framework to handle such cases by imposing specific relations on the tableau entries.

\section{Relation Gelfand--Tsetlin modules}\label{sec:RelGTmodules}

In \cite{FRZ19}, the class of relation Gelfand--Tsetlin modules was introduced as an attempt to unify several known constructions of Gelfand--Tsetlin modules with diagonalizable action of $\Gamma$ (see \cite{GG65, LP79, Maz98, Maz03}). This section is devoted to describing the construction and main properties of relation Gelfand--Tsetlin modules.
\subsection{Relation graphs}
 To unify the various constructions of Gelfand--Tsetlin modules, we require a combinatorial device that encodes which tableau entries are constrained by integer differences and which are free. We achieve this by associating a directed graph $G$ with the set of tableau coordinates. The edges of $G$ will dictate specific integrality conditions on the entries, effectively determining the skeleton of the resulting module.

Denote by $\mathfrak{V}$ the set $\{(i,j)\ |\ 1\leq j\leq i\leq n\}$ arranged in a triangular configuration with $n$ rows, where the $k$-th row is written as $((k,1),\ldots,(k,k))$, and top row given by the $n$-th row. From now on, we consider only directed graphs $G$ with set of vertices $\mathfrak{V}$. As we will see later in this section, we will consider certain Gelfand--Tsetlin tableaux associated with a given graph $G$, and the following definition contains all the necessary conditions for those tableaux to define a Gelfand--Tsetlin module.
\begin{figure}[!htbp]
\begin{center}
\begin{tabular}{c c}
\xymatrixrowsep{0.5cm}
\xymatrixcolsep{0.1cm}\xymatrix @C=0.1em { 
 \scriptstyle{(4,1)}& &\scriptstyle{(4,2)}  & &\scriptstyle{(4,3)} &   &\scriptstyle{(4,4)} \\
  & \scriptstyle{(3,1)}   & &\scriptstyle{(3,2)}  & &\scriptstyle{(3,3)}  \\
   &     &\scriptstyle{(2,1)} & &\scriptstyle{(2,2)} & \\
 &    & &\scriptstyle{(1,1)}  & &\\
}\\
\end{tabular}
\end{center}
\caption{Configuration of the set of vertices $\mathfrak{V}$ for $n=4$}
\end{figure}
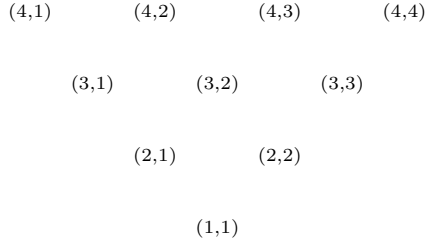

\begin{definition}
For any  $1\leq i<j\leq k\leq n-1$, we call $((k,i),(k,j))$ an \emph{adjoining pair of vertices in $G$} if there is a directed path from $(k,i)$ to $(k,j)$ and whenever $ i<t< j$, there are no directed paths from $(k,i)$ to $(k,t)$ or from $(k,t)$ to $(k,j)$. $G$ is called a \emph{relation graph} if the following conditions are satisfied:

\begin{itemize}

\item[(i)]  The only possible arrows in $G$ are the ones connecting vertices in consecutive rows or vertices in the $n$-th row.
\item[(ii)] For any $1\leq k\leq n$ and $1\leq i<j\leq k$, there are no directed paths from $(k,j)$ to $(k,i)$. 
\item[(iii)] $G$ does not contain oriented cycles or multiple arrows.

\item[(iv)] For any two vertices in the same row $k<n$ and the same connected component (of the unoriented graph associated to $G$), there exists an oriented path in $G$ from one vertex to the other.

\item[(v)] G does not contain a subgraph of the form:
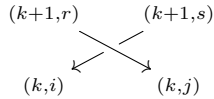
\begin{center}
\begin{tabular}{c c }
\xymatrixrowsep{0.5cm}
\xymatrixcolsep{0.1cm}
\xymatrix @C=0.2em{
 \scriptstyle{(k+1,r)}\ar[rrrd] &   & &  \scriptstyle{(k+1,s)}\ar[llld] |!{[d];[lll]}\hole    & \\
 \scriptstyle{(k,i)}& &    &\scriptstyle{(k,j)}   & }
\end{tabular}
\end{center}
 with $1\leq r<s\leq k+1$, and $1\leq i< j\leq k$.

 \item[(vi)] For every adjoining pair of vertices $((k,i),(k,j))$, one of the following is a subgraph of $G$: 
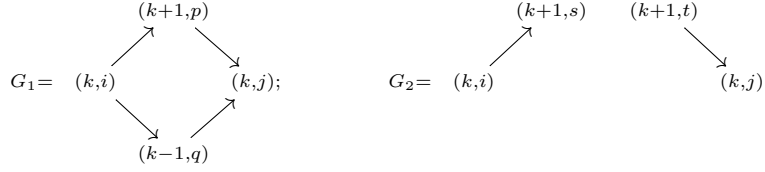
\begin{center}
\begin{tabular}{c c c c}
\xymatrixrowsep{0.5cm}
\xymatrixcolsep{0.1cm}
\xymatrix @C=0.2em{
  &   &\scriptstyle{(k+1,p)}\ar[rd]   &   & \\
 \scriptstyle{G_1}=  &\scriptstyle{(k,i)}\ar[rd] \ar[ru]  &    &\scriptstyle{(k,j)};   &  \\
   &   &\scriptstyle{(k-1,q)}\ar[ru]   &   & }
&\ \ &
\xymatrixrowsep{0.5cm}
\xymatrixcolsep{0.1cm}\xymatrix @C=0.2em {
   &   &\scriptstyle{(k+1,s)}    &   &\scriptstyle{(k+1,t)}\ar[rd]&& \\
  \scriptstyle{G_{2}}= &\scriptstyle{(k,i)} \ar[ru]  & &   & & \scriptstyle{(k,j)} \\
   &   &   &   & &&}
\end{tabular}
\end{center}
for some $1\leq q\leq k-1$, $1\leq p\leq k+1$, or $1\leq s<t\leq k+1.$
    
\end{itemize}
\end{definition}

We write vectors in $\mathbb{C}^{\frac{n(n+1)}{2}}$ as ordered tuples $L=(l_{n1},\ldots,l_{nn}|\ldots|l_{22},l_{21}|l_{11})$ indexed by elements in $\mathfrak{V}$, and by $T(L)$ we denote the Gelfand--Tsetlin tableau with $l_{ij}$ in the position of the vertex $(i,j)$ of $\mathfrak{V}$.

 With the graph structure established, we can now define the class of tableaux associated with it. A tableau will be considered a realization of the graph $G$ if its entries satisfy integer difference conditions corresponding to the edges of $G$. These conditions ensure that the resulting module avoids singularities while preserving the necessary relations for the Gelfand--Tsetlin action.

\begin{definition}\label{def:satisfrealization} Let $G$ be any graph, $T(L)$ any Gelfand--Tsetlin tableau, and 
\[{\mathbb Z}_0^\frac{n(n+1)}{2}:=\{w\in {\mathbb Z}^\frac{n(n+1)}{2}\ : \ w_{n,i}=0 \text{ for any } 1\leq i\leq n\}.
\]

\begin{itemize}
\item[(i)] We say that \emph{$T(L)$ satisfies $G$} if
\begin{itemize}
\item[(a)] $l_{ij}-l_{rs}\in \mathbb{Z}_{\geq 0}$ whenever $(i,j)$ and $(r,s)$ are connected by a horizontal arrow, or an arrow pointing down.
\item[(b)] $l_{ij}-l_{rs}\in \mathbb{Z}_{> 0}$ whenever $(i,j)$ and $(r,s)$ are connected by an arrow pointing up.
\end{itemize}
\item[(ii)] We say that \emph{$T(L)$ is a $G$-realization} if
\begin{itemize}
\item[(a)] $T(L)$ satisfies $G$.
\item[(b)]  For any $1\leq k\leq n-1$, we have $l_{ki}-l_{kj}\in \mathbb{Z} $ only if $(k,i)$ and $(k,j)$ are in the same connected component of the unoriented graph associated to $G$. 
\end{itemize}

\item[(iii)] If $T(L)$ is a $G$-realization, by ${\mathcal B}_{G}(T(L))$ we denote the set of all $G$-realizations of the form $T(L+z)$, with $z\in {\mathbb Z}_0^\frac{n(n+1)}{2}$.
\end{itemize}
\end{definition}

\begin{example}
Consider the following Gelfand--Tsetlin tableaux:

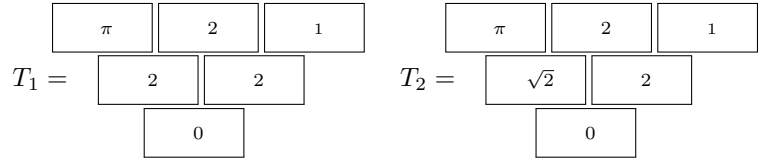
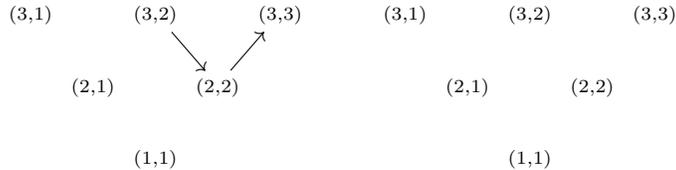
\begin{center}
			\hspace{1.4cm}\Stone{\mbox{ \scriptsize {$ \pi$}}}\Stone{\mbox{ \scriptsize {$2$}}}\Stone{\mbox{ \tiny {$1$}}}\hspace{1cm}\Stone{\mbox{ \scriptsize {$\pi$}}}\Stone{\mbox{ \scriptsize {$2$}}}\Stone{\mbox{ \tiny {$1$}}}\\
			$T_1=$\hspace{0.4cm}\Stone{\mbox{ \scriptsize {$2$}}}\Stone{\mbox{ \scriptsize {$2$}}}\hspace{1.2cm}$T_2=$\hspace{0.4cm}\Stone{\mbox{ \scriptsize {$\sqrt{2}$}}}\Stone{\mbox{ \scriptsize {$2$}}}\\[0.2pt]
			\hspace{1cm}\Stone{\mbox{ \scriptsize {$0$}}}\hspace{3.8cm}\Stone{\mbox{ \scriptsize {$0$}}}
	\end{center}
and the graphs:
\begin{center}
\begin{tabular}{c c c c c}
\xymatrixrowsep{0.5cm}
\xymatrixcolsep{0.1cm}\xymatrix @C=0.1em {
  & \scriptstyle{(3,1)} & &\scriptstyle{(3,2)}\ar[rd]  & &\scriptstyle{(3,3)}  & \ \ \ \ \   & \scriptstyle{(3,1)} & &\scriptstyle{(3,2)}  & &\scriptstyle{(3,3)} \\
    &     &\scriptstyle{(2,1)}& &\scriptstyle{(2,2)}\ar[ru]  & &  \ \ \ \ \    &     &\scriptstyle{(2,1)}& &\scriptstyle{(2,2)}\\
  &    & &\scriptstyle{(1,1)}& & &  \ \ \ \ \  &    & &\scriptstyle{(1,1)}& &\\
}
\end{tabular}

\end{center}
In this case, $T_1$ satisfies both graphs, but is not a realization of either graph. On the other hand, $T_2$ is a realization of both graphs.

\end{example}

\begin{theorem}[Theorem 4.33, and Theorem 5.8 \cite{FRZ19}]
For any relation graph $G$, and any $G$-realization $T(L)$, the vector space $V_{G}(T(L))$ has the structure of a $\mathfrak{gl}(n)$-module, with the action of $\mathfrak{gl}(n)$ given by the Gelfand--Tsetlin formulas~\eqref{eq:GTformulas}. Moreover, $V_G(T(L))$ is a Gelfand--Tsetlin module with diagonalizable action of the generators of $\Gamma$, and the dimension of $V_G(T(L))(\chi)$ is equal to 1 for any $\chi\in\Gamma^{*}$ in the support of $V_G(T(L))$.
\end{theorem}

\begin{definition}\label{def:relmod}
Modules isomorphic to $V_{G}(T(L))$ for some relation graph $G$ will be called \emph{relation modules}.
\end{definition}
\begin{example}\label{ex:families}
Below, we consider the graph associated with some families of relation modules that were constructed in previous works. Let us fix $n=4$. 

\noindent
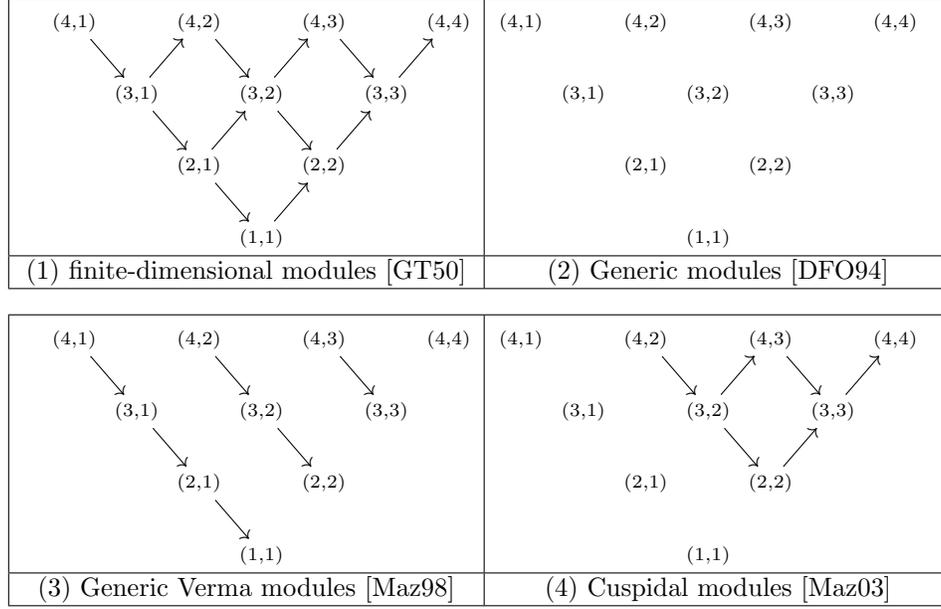
\resizebox{\linewidth}{!}{%
\begin{tabular}{|c|c|}
\hline
\begin{tabular}{c}
\xymatrixrowsep{0.45cm}
\xymatrixcolsep{0cm}
\xymatrix @C=0.08em {
 \scriptstyle{(4,1)}\ar[rd]& &\scriptstyle{(4,2)} \ar[rd]   & &\scriptstyle{(4,3)}\ar[rd] &   &\scriptstyle{(4,4)} \\
  & \scriptstyle{(3,1)} \ar[ru] \ar[rd]   & &\scriptstyle{(3,2)}\ar[rd]\ar[ru]  & &\scriptstyle{(3,3)} \ar[ru] \\
    &     &\scriptstyle{(2,1)}\ar[ru] \ar[rd] & &\scriptstyle{(2,2)}\ar[ru]  & \\
  &    & &\scriptstyle{(1,1)} \ar[ru] & &\\
}
\end{tabular}
&
\begin{tabular}{c}
\xymatrixrowsep{0.45cm}
\xymatrixcolsep{0cm}
\xymatrix @C=0.08em {
 \scriptstyle{(4,1)}& &\scriptstyle{(4,2)}  & &\scriptstyle{(4,3)} &   &\scriptstyle{(4,4)} \\
  & \scriptstyle{(3,1)}   & &\scriptstyle{(3,2)}  & &\scriptstyle{(3,3)}  \\
    &     &\scriptstyle{(2,1)} & &\scriptstyle{(2,2)} & \\
&    & &\scriptstyle{(1,1)}  & &\\
}
\end{tabular}
\\
\hline
$(1)$ finite-dimensional modules \cite{GT50} & $(2)$ Generic modules \cite{DFO94}\\
\hline
\end{tabular}%
}
\noindent
\resizebox{\linewidth}{!}{%
\begin{tabular}{|c|c|}
\hline
\begin{tabular}{c}
\xymatrixrowsep{0.45cm}
\xymatrixcolsep{0.08cm}
\xymatrix @C=0.08em { 
 \scriptstyle{(4,1)}\ar[rd]& &\scriptstyle{(4,2)} \ar[rd]   & &\scriptstyle{(4,3)}\ar[rd] &   &\scriptstyle{(4,4)} \\
& \scriptstyle{(3,1)}  \ar[rd]   & &\scriptstyle{(3,2)}\ar[rd]  & &\scriptstyle{(3,3)}  \\
    &     &\scriptstyle{(2,1)} \ar[rd] & &\scriptstyle{(2,2)}  & \\
 &    & &\scriptstyle{(1,1)} & &\\
}
\end{tabular}
&
\begin{tabular}{c}
\xymatrixrowsep{0.45cm}
\xymatrixcolsep{0.08cm}
\xymatrix @C=0.08em { 
 \scriptstyle{(4,1)}& &\scriptstyle{(4,2)} \ar[rd]  & &\scriptstyle{(4,3)}\ar[rd] &   &\scriptstyle{(4,4)} \\
  & \scriptstyle{(3,1)}   & &\scriptstyle{(3,2)}\ar[rd]\ar[ru]  & &\scriptstyle{(3,3)} \ar[ru] \\
    &     &\scriptstyle{(2,1)} & &\scriptstyle{(2,2)}\ar[ru]  & \\
  &    & &\scriptstyle{(1,1)}  & &\\
}
\end{tabular}
\\
\hline
$(3)$ Generic Verma modules \cite{Maz98} & $(4)$ Cuspidal modules \cite{Maz03}\\
\hline
\end{tabular}%
}

Although the relation modules obtained using the graph in $(4)$ are cuspidal modules, this family does not exhaust all cuspidal modules.
\end{example}

\begin{remark}\label{rem:separation}
The Gelfand--Tsetlin subalgebra $\Gamma$ separates basis elements of $V_G(T(L))$ (see \cite[Theorem 5.8]{FRZ19}), which means that, given tableaux $T(Q)\ne T(R)$ in $\mathcal{B}_{G}(T(L))$, there exists $\gamma\in\Gamma$ such that $\gamma\cdot T(Q)=0$ and $\gamma\cdot T(R)= T(R)$.
\end{remark}

\subsection{Useful definitions}\label{sec:definitions}
 Having constructed the relation modules $V_G(T(L))$, our main goal is to analyze their internal structure, specifically their simple subquotients. Since the submodule structure is governed by the placement of integer differences in the tableau (which cause certain coefficients in the Gelfand--Tsetlin formulas to vanish), we need a finer combinatorial invariant than the graph $G$ itself. In this section, we introduce the graph $G(L)$ associated with a specific tableau $T(L)$, which encodes the actual integer relations present in a given basis vector.

\begin{definition}\label{def:graphG(L)} Given a Gelfand--Tsetlin tableau $T(L)$, denote by $G(L)$ the graph with set of vertices $\mathfrak{V}$ and an arrow from $(i,j)$ to $(r,s)$ if: 
 \begin{itemize}
 \item[(i)] $i=r+1$, and $l_{ij}-l_{rs}\in\mathbb{Z}_{\geq 0}$; or 
 \item[(ii)] $i=r-1$, and $l_{ij}-l_{rs}\in\mathbb{Z}_{>0}$; or
 \item[(iii)] $i=r=n$, $j< s$, and $l_{ij}-l_{rs}\in\mathbb{Z}_{\geq 0}$.
\end{itemize}

 Moreover, given a relation graph $G$, denote by $\overline{G}$ the graph with set of vertices $\mathfrak{V}$ and an arrow from $(i,j)$ to $(r,s)$ if there is a directed path in $G$ from $(i,j)$ to $(r,s)$ and $|i-r|=1$; or $i=r=n$, and $j\neq s$. Finally, for any graph $H$, denote by $E_H$ the set of all arrows of $H$, and by $E_{H}^+$ the set of arrows in $H$ pointing down. 
\end{definition}

\begin{example}
Let $G$ and $T(R)$ denote the following graph and tableau:

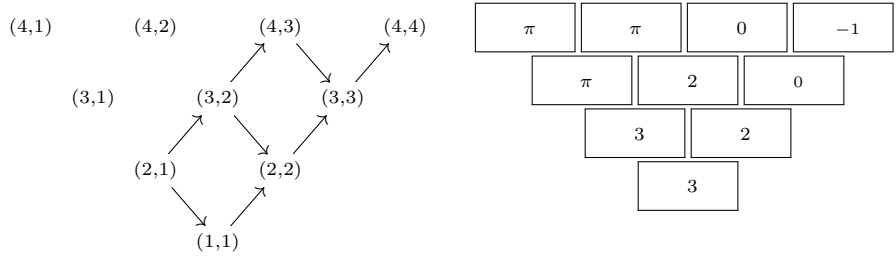
\columnratio{0.55}

\begin{paracol}{2}

  \begin{tabular}{c c}
\xymatrixrowsep{0.5cm}
\xymatrixcolsep{0.1cm}\xymatrix @C=0.1em {
 \scriptstyle{(4,1)}& &\scriptstyle{(4,2)}    & &\scriptstyle{(4,3)}\ar[rd] &   &\scriptstyle{(4,4)} \\
  & \scriptstyle{(3,1)}    & &\scriptstyle{(3,2)}\ar[rd]\ar[ru]  & &\scriptstyle{(3,3)}\ar[ru]  \\
    &     &\scriptstyle{(2,1)}\ar[ru] \ar[rd] & &\scriptstyle{(2,2)}\ar[ru]  & \\
  &    & &\scriptstyle{(1,1)} \ar[ru] & &\\
}
\end{tabular}
  \switchcolumn
  \begin{center}
  \hspace{-0.2cm}
  \Stone{\mbox{ \scriptsize {$\pi$}}}\Stone{\mbox{ \scriptsize {$\pi$}}}\Stone{\mbox{ \scriptsize {$0$}}}\Stone{\mbox{ \tiny {$-1$}}}\\[0.2pt]
\Stone{\mbox{ \scriptsize {$\pi$}}}\Stone{\mbox{ \scriptsize {$2$}}}\Stone{\mbox{ \tiny {$0$}}}\\[0.2pt]
\Stone{\mbox{ \scriptsize {$3$}}}\Stone{\mbox{ \scriptsize {$2$}}}\\[0.2pt]
\Stone{\mbox{ \scriptsize {$3$}}}\\
\end{center}
\end{paracol}
The graphs $\overline{G}$ and $G(R)$ are, respectively,

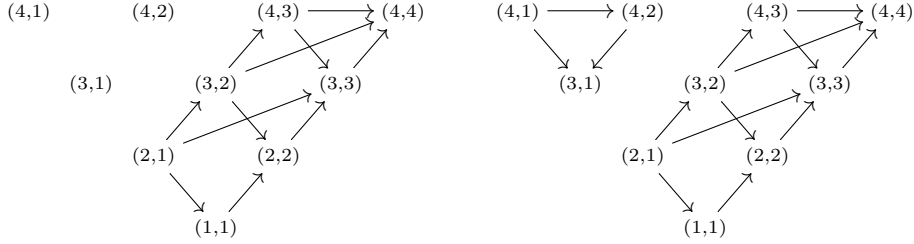
\columnratio{0.5}
 
\begin{paracol}{2}
\begin{center}
\hspace{-0.3cm}
  \begin{tabular}{c c}
\xymatrixrowsep{0.5cm}
\xymatrixcolsep{0.1cm}\xymatrix @C=0.1em {
 \scriptstyle{(4,1)}& &\scriptstyle{(4,2)}    & &\scriptstyle{(4,3)}\ar[rd] \ar[rr]&   &\scriptstyle{(4,4)} \\
  & \scriptstyle{(3,1)}  & &\scriptstyle{(3,2)}\ar[rd]\ar[ru] \ar[rrru] & &\scriptstyle{(3,3)}\ar[ru]  \\
    &     &\scriptstyle{(2,1)}\ar[ru] \ar[rrru]\ar[rd] & &\scriptstyle{(2,2)}\ar[ru]  & \\
  &    & &\scriptstyle{(1,1)} \ar[ru] & &\\
}
\end{tabular}

\end{center}
  \switchcolumn
\begin{center}
\hspace{-0.3cm}
  \begin{tabular}{c c}
\xymatrixrowsep{0.5cm}
\xymatrixcolsep{0.1cm}\xymatrix @C=0.1em {
 \scriptstyle{(4,1)}\ar[rr]\ar[rd]& &\scriptstyle{(4,2)}\ar[dl]    & &\scriptstyle{(4,3)}\ar[rd]\ar[rr] &   &\scriptstyle{(4,4)} \\
  & \scriptstyle{(3,1)}   & &\scriptstyle{(3,2)}\ar[rd]\ar[ru] \ar[urrr] & &\scriptstyle{(3,3)}\ar[ru]  \\
    &     &\scriptstyle{(2,1)}\ar[ru] \ar[rrru]\ar[rd] & &\scriptstyle{(2,2)}\ar[ru]  & \\
  &    & &\scriptstyle{(1,1)} \ar[ru] & &\\
}
\end{tabular}
\end{center}
\end{paracol}
\end{example}

\section{Simple subquotients}\label{sec:simplesubquotients}

The main results of this paper establish a basis for both the submodule generated by a tableau $T(L)$ that is a $G$-realization of some relation graph $G$ and for the irreducible subquotient containing a given $G$-realization.

In \cite{FGR15}, associated with any Gelfand-Tsetlin tableau $T(Q)$, the set of triples given by $\Omega^{+}(T(Q)):=\{(r,s,t)\ |\ q_{rs}-q_{r-1,t}\in\mathbb{Z}_{\geq 0}\}$ was crucial in order to describe the structure of the simple subquotients of generic modules. In the context of relation modules, it is convenient to consider elements of $\Omega^{+}(T(Q))$ as arrows in $G(Q)$.

\begin{remark}
If $G$ is a relation graph, and $T(Q)$ is a $G$-realization, the map sending $(r,s,t)\in \Omega^{+}(T(Q))$ to the arrow in $G(Q)$ from $(r,s)$ to $(r-1,t)$ defines a one-to-one correspondence between the sets $\Omega^{+}(T(Q))$ and $E_{G(Q)}^+$.
\end{remark}
\begin{remark}
As our graphs are induced by the integral relations satisfied by the entries of the tableaux, we order the vertices of the graph accordingly, namely, if there is an arrow from $(a,b)$ to $(c,d)$, we define $(a,b)>(c,d)$.
\end{remark}

\begin{definition}
Given a graph $G$, let
$z\in\mathbb{Z}_0^{\frac{n(n+1)}{2}}\setminus\{0\}.$
A \emph{maximal $G$-chain with respect to $z$} is the ordered set of vertices in an oriented path in $G$ satisfying:
\begin{itemize}
\item[(i)] For any vertex $(i,j)$ in the chain, we have $z_{ij}\neq 0$. 
\item[(ii)] If $(a,b)$ is the maximum element of the path, and there is an arrow in $G$ from $(r,s)$ to $(a,b)$, then $z_{rs}=0$.
\item[(iii)] If $(c,d)$ is the minimum element of the path, and there is an arrow in $G$ from $(c,d)$ to $(r,s)$, then $z_{rs}=0$.
\end{itemize}
\end{definition}

\begin{remark}
For any graph $G$ without cycles, and $z\in\mathbb{Z}_0^{\frac{n(n+1)}{2}}\setminus\{0\}$ with $z_{ab}\neq 0$, there exists a maximal $G$-chain with respect to $z$ containing $(a,b)$.
\end{remark}

\begin{example}\label{ex:diamantecomp}
 Let $G$ and $T(R)$ denote the following graph and tableau:

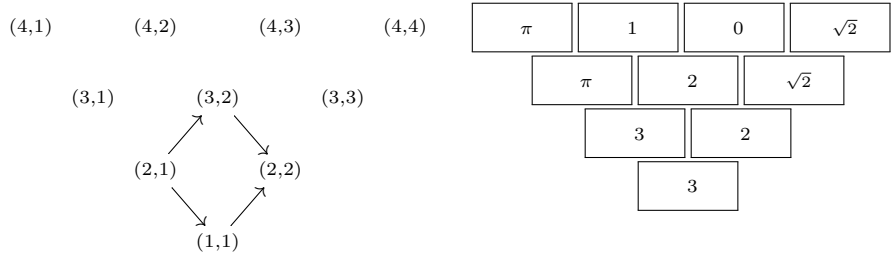
\columnratio{0.55}

\begin{paracol}{2}

  \begin{tabular}{c c}
\xymatrixrowsep{0.45cm}
\xymatrixcolsep{0.1cm}\xymatrix @C=0.1em {
 \scriptstyle{(4,1)}& &\scriptstyle{(4,2)}    & &\scriptstyle{(4,3)} &   &\scriptstyle{(4,4)} \\
  & \scriptstyle{(3,1)}    & &\scriptstyle{(3,2)}\ar[rd]  & &\scriptstyle{(3,3)}  \\
    &     &\scriptstyle{(2,1)}\ar[ru] \ar[rd] & &\scriptstyle{(2,2)}  & \\
  &    & &\scriptstyle{(1,1)} \ar[ru] & &\\
}
\end{tabular}
  \switchcolumn
  \begin{center}
  \hspace{-0.3cm}
  \Stone{\mbox{ \scriptsize {$\pi$}}}\Stone{\mbox{ \scriptsize {$1$}}}\Stone{\mbox{ \scriptsize {$0$}}}\Stone{\mbox{ \tiny {$\sqrt{2}$}}}\\[0.2pt]
\Stone{\mbox{ \scriptsize {$\pi$}}}\Stone{\mbox{ \scriptsize {$2$}}}\Stone{\mbox{ \tiny {$\sqrt{2}$}}}\\[0.2pt]
\Stone{\mbox{ \scriptsize {$3$}}}\Stone{\mbox{ \scriptsize {$2$}}}\\[0.2pt]
\Stone{\mbox{ \scriptsize {$3$}}}\\
\end{center}
\end{paracol}

If $z=(0,0,0,0|-1,-1,-1|1,-1|-1)$, the maximal $G$-chains with respect to $z$ are $\mathcal{C}=\{(2,1),(1,1),(2,2)\}$, and $\mathcal{C}'=\{(2,1),(3,2),(2,2)\}$. Moreover, the graph $G(R)$ is given by:

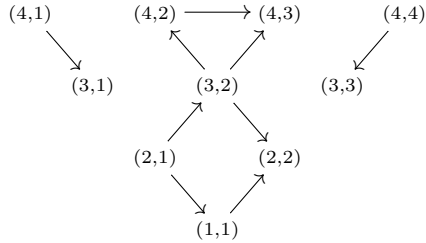
\begin{center}
\begin{tabular}{c c}
\xymatrixrowsep{0.5cm}
\xymatrixcolsep{0.1cm}\xymatrix @C=0.1em {
 \scriptstyle{(4,1)}\ar[rd]& &\scriptstyle{(4,2)}  \ar[rr]  & &\scriptstyle{(4,3)} &   &\scriptstyle{(4,4)} \ar[ld]\\
  & \scriptstyle{(3,1)}    & &\scriptstyle{(3,2)}\ar[rd]\ar[lu]\ar[ru]  & &\scriptstyle{(3,3)}  \\
    &     &\scriptstyle{(2,1)}\ar[ru] \ar[rd] & &\scriptstyle{(2,2)}  & \\
  &    & &\scriptstyle{(1,1)} \ar[ru] & &\\\\
}
\end{tabular}
\end{center}
and the maximal $G(R)$-chains with respect to $z$ are $\mathcal{C}$, $\mathcal{C}'$, $\{(3,1)\}$, and $\{(3,3)\}$.
\end{example}

From now on and until the end of this paper, $G$ will denote a   relation graph, and $T(L)$ any $G$-realization. The next result generalizes \cite[Lemma 6.3]{FGR15} and is the key ingredient for the description of bases of cyclic submodules of $V_G(T(L))$ (see Proposition \ref{prop:translgerada} below).

\begin{proposition}\label{prop:general} Let $T(R)\in\mathcal{B}_G(T(L))$, and $z\in\mathbb{Z}^{\frac{n(n+1)}{2}}_0\setminus\{0\}$ such that $T(R+z)$ is a $G$-realization, and $E^+_{G(R)}\subseteq E^+_{G(R+z)}$. Then there exists $(i,j)\in\mathfrak{V}$ such that $z_{ij}\ne 0$ and $
E^+_{G(R)}\subseteq E^+_{G(R+z_{ij}\delta^{ij})} \subseteq E^+_{G(R+z)}.$

\end{proposition}
\begin{proof}
We prove the statement by induction on
\[
t(z):=\Bigl|\{(a,b)\in \mathfrak{V}\mid z_{ab}\ne 0\}\Bigr|.
\]
The case $t(z)=1$ is trivial.

Suppose now that $t(z)>1$, and consider $C=\{(a_1,b_1),\ldots,(a_t,b_t)\}$ a maximal $G(R)$-chain with respect to $z$.
\begin{description}
\item[Case 1.] If $z_{a_1,b_1}>0$, as $T(R)$ and $T(R+z)$ are $G$-realizations, it is straightforward to check that we can consider $(i,j)=(a_1,b_1)$.

\item[Case 2.] If $z_{a_1,b_1}<0$, let $w=z-z_{a_1,b_1}\delta^{a_1,b_1}$. In this case, the tableau $T(R+w)$ is a $G$-realization, since $C$ is a maximal $G(R)$-chain and
\[
r_{a_1,b_1}-r_{a_2, b_2}-z_{a_2, b_2}\ge -z_{a_1,b_1}>0.
\]
In addition, $t(w)=t(z)-1$ and $E^+_{G(R+w)}\subseteq E^+_{G(R+z)}$. By the induction hypothesis, there is $(i,j)\in\mathfrak{V}$ such that $z_{i,j}=w_{ij}\ne 0$, and \[
E^+_{G(R)}\subseteq E^+_{G(R+w_{ij}\delta^{ij})}\subseteq E^+_{G(R+w)}\subseteq E^+_{G(R+z)}.
\qedhere
\]
\end{description}
\end{proof} 

 The existence of the chains provides the precise mechanism to navigate between tableaux. If a valid chain connects specific vertices, it ensures that the denominators in the Gelfand--Tsetlin formulas do not cause the action to vanish, allowing us to transition between corresponding basis vectors. We now formalize this notion of reachability into a partial order on the set of tableaux.
 
\begin{definition} Given $T(Q), T(R)\in \mathcal{B}_{G}(T(L))$, we write $T(R)\preceq_{(1)}T(Q)$ if there exists $g\in\mathfrak{gl}(n)$ such that $T(Q)$ appears with nonzero coefficient in the decomposition of $g\cdot T(R)$ as a linear combination of tableaux. For any $p\ge 1$, we write $T(R)\preceq_{(p)} T(Q)$ if there exist tableaux $T(R^{(0)})$, $T(R^{(1)})$, \ldots, $T(R^{(p)})$, such that 
\[
T(R)=T(R^{(0)})\preceq_{(1)} T(R^{(1)})\preceq_{(1)}\cdots\preceq_{(1)} T(R^{(p)})=T(Q).
\]
\end{definition}

As an immediate consequence of the definition of $\preceq_{(p)}$ we have the following

\begin{lemma}[Lemma 6.6 \cite{FGR15}]\label{lem:preceq}
If $T(Q^{(0)})$, $T(Q^{(1)})$, and $T(Q^{(2)})$ are tableaux in $\mathcal{B}_{G}(T(L))$,
where $G$ is a relation graph and $T(L)$ is a $G$-realization, then
\[
T(Q^{(0)})\preceq_{(p)}T(Q^{(1)}) \text{ and } T(Q^{(1)})\preceq_{(q)}T(Q^{(2)})
\quad\Longrightarrow\quad
T(Q^{(0)})\preceq_{(p+q)}T(Q^{(2)})
\]
for some $p,q \in \mathbb{Z}_{\ge 1}$.
\end{lemma}

\begin{corollary}\label{cor:gerado}
If $T(R),T(Q)\in \mathcal{B}_{G}(T(L))$, and $T(R)\preceq_{(p)} T(Q)$ for some $p$, then 
\begin{center}$T(Q)\in U\cdot T(R)$.\end{center}
\end{corollary}
\begin{proof}
Use the separation argument described in Remark \ref{rem:separation} to show that $T(R)\preceq_{(1)} T(Q)$ implies $T(Q)\in U\cdot T(R)$. Then, use Lemma \ref{lem:preceq} to show that $T(R)\preceq_{(p)}T(Q)$ implies $T(Q)\in U\cdot T(R)$.
\end{proof}

\begin{proposition}\label{prop:translgerada}
Let $z\in\mathbb{Z}^{\frac{n(n+1)}{2}}_0\setminus\{0\}$ such that $T(R)$ and $T(R+z)$ are $G$-realizations, and $E^+_{G(R)}\subseteq E^+_{G(R+z)}$. Then $T(R+z)\in U\cdot T(R)$.

\end{proposition}
\begin{proof}
Let $t(z)$ denote the number of non-zero entries of $z$. By Corollary \ref{cor:gerado}, it is enough to prove that $T(R)\preceq_{(p)} T(R+z)$ for some $p\geq 1$. Let us prove the statement by induction on $t(z)$. The case $t=1$ is a consequence of the fact that $T(R+l\delta^{ij})\preceq_{(1)} T(R+(l+1)\delta^{ij})$ for any $0\leq l\leq  z_{ij}$, with $z_{ij}\ne 0$. The inductive step uses Proposition \ref{prop:general} and the fact that whenever $(i,j)\in\mathfrak{V}$ is such that $z_{ij}\ne 0$ and
\[
E^+_{G(R)}\subseteq E^+_{G(R+z_{ij}\delta^{ij})} \subseteq E^+_{G(R+z)},
\]
the vector $w=z-z_{ij}\delta^{ij}$ satisfies $t(w)=t(z)-1$, and the inductive hypothesis implies the existence of $p,q\in\mathbb{Z}_{\ge 0}$ such that
\[
T(R)\preceq_{(p)} T(R+w)\quad \textup{and}\quad T(R+w)\preceq_{(q)} T(R+z)
\]
and hence, by Lemma \ref{lem:preceq}, $T(R)\preceq_{(p+q)} T(R+z)$. This concludes the proof.
\end{proof}

 With the partial order established, we can now characterize the submodules. A submodule generated by a tableau $T(R)$ consists precisely of all tableaux reachable from $T(R)$ (i.e., those that are greater than or equal to $T(R)$ in the specific ordering induced by the graph structure). This geometric intuition translates into the following algebraic statement regarding the basis of the submodule.

\begin{theorem}\label{theo:main1}
Let $T(R)\in \mathcal{B}_{G}(T(L))$.
A basis for the $U$-submodule of $V_{G}(T(L))$ generated by $T(R)$ is given by
\[
\mathcal{N}_G(T(R))
:=\{\,T(S)\in \mathcal{B}_{G}(T(L)) \mid E^+_{G(R)}\subseteq E^+_{G(S)}\,\}.
\]
\end{theorem}

\begin{proof}
Denote by $W_G(T(R))$ the linear span of $\mathcal{N}_G(T(R))$. As in the proof of \cite[Theorem~6.8]{FGR15}, one must show that $W_G(T(R))$ is a $\mathfrak{gl}(n)$-submodule of $V_G(T(L))$, and hence we have $U\cdot T(R)\subseteq W_G(T(R))$. To prove that $W_G(T(R))\subseteq U\cdot T(R)$, we use Proposition~\ref{prop:translgerada}.
\end{proof}

\begin{corollary}
Let $T(R)$ and $T(Q)$ be in $\mathcal{B}_{G}(T(L))$. Then $U\cdot T(R)=U\cdot T(Q)$ if and only if $E^+_{G(R)}=E^+_{G(Q)}$.
\end{corollary}
\begin{proof}
This follows directly from Theorem~\ref{theo:main1}.
\end{proof}

\begin{corollary}\label{cor:geradoigualrelation}
Let $T(R)$ be in $\mathcal{B}_{G}(T(L))$. If $E^+_{G(R)}=E^+_{\overline{G}}$, then $U\cdot T(R)=V_{G}(T(L))$. 
\end{corollary}
\begin{proof}
Since $E^+_{G(R)}=E^+_{\overline{G}}$, for any $T(Q)\in\mathcal{B}_G(T(L))$ we have $E^+_{\overline{G}}\subseteq E^+_{G(Q)}$, which implies $T(Q)\in\mathcal{N}_G(T(R))$. Finally, by Theorem \ref{theo:main1}, $U\cdot T(R)=V_G(T(L))$.
\end{proof}

Finally, we address the classification of simple subquotients. A simple subquotient arises from the difference between a submodule and its maximal proper submodules. In our combinatorial framework, this corresponds to identifying sets of tableaux that are equivalent in terms of their generating power, specifically, those that share the exact same configuration of downward arrows.

\begin{theorem}\label{theo:main2}
Let $T(R)\in \mathcal{B}_{G}(T(L))$.
A basis for the simple subquotient of $V_{G}(T(L))$ containing $T(R)$ is given by
\[
\mathcal{I}_G(T(R))
:=\{\,T(S)\in \mathcal{B}_{G}(T(L)) \mid
E^+_{G(R)}=E^+_{G(S)}\,\}.
\]
\end{theorem}

\begin{proof}
As in the proof of \cite[Theorem 6.14]{FGR15}, for each tableau $T(R)$ in $\mathcal{B}_{G}(T(L))$, we consider the module containing $T(R)$ given as follows
$$
M(T(R)):=U\cdot T(R)\left/\sum U\cdot T(Q)\right.,
$$
where the sum is taken over tableaux $T(Q)$ such that $T(Q)\in U\cdot T(R)$, and $U\cdot T(Q)$ is a proper submodule of
$U\cdot T(R)$. Since
\[
{E^+_{\overline{G}}\subset \bigcap_{T(Q)\in\mathcal{B}_G(T(L))}} E^+_{G(Q)},
\]
the structure of $M(T(R))$ depends only on $E^+_{G(R)}\backslash E^+_{\overline{G}}$. Indeed, we have $E^+_{G(Q)}\subseteq E^+_{G(S)}$ if and only if $(E^+_{G(Q)}\backslash E^+_{\overline{G}})\subseteq (E^+_{G(S)}\backslash E^+_{\overline{G}})$. The simplicity of $M(T(R))$ follows from the fact that any nonzero tableau $T(S)$ in $M(T(R))$ is such that $U\cdot T(S)=U\cdot T(R)$, which implies that $T(S)$ generates $M(T(R))$. Finally, a basis for $M(T(R))$ is given by $\mathcal{N}_G(T(R))$ excluding the basis elements in $\sum U\cdot T(Q)$, which by Theorem~\ref{theo:main1}, is given by $\left\{T(S)\,:\,E^+_{G(R)}\subsetneq E^+_{G(S)}\right\}.$
Therefore, $\mathcal{I}_G(T(R))$ is a~basis for $M(T(R))$.
\end{proof}

 \section{Examples}\label{sec:examples}

\begin{example}
Here we consider a tableau $T(L)$, a graph $G$, and the graph $G(L)$.

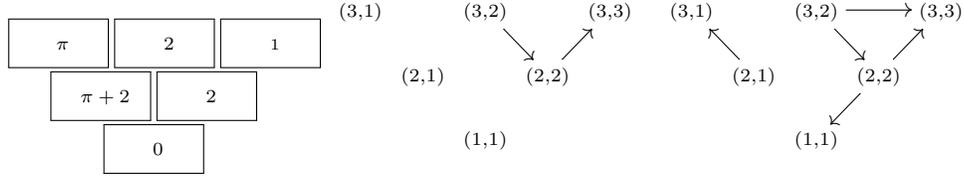
\columnratio{0.33}

\begin{paracol}{3}
\begin{center}
\Stone{\mbox{ \scriptsize {$\pi$}}}\Stone{\mbox{ \scriptsize {$2$}}}\Stone{\mbox{ \tiny {$1$}}}\\[0.2pt]
\Stone{\mbox{ \scriptsize {$\pi+2$}}}\Stone{\mbox{ \scriptsize {$2$}}}\\[0.2pt]
\Stone{\mbox{ \scriptsize {$2$}}}
\end{center}
\hspace{-0.8cm}
\switchcolumn
\hspace{-0.8cm}
\begin{tabular}{c c c c c}
\xymatrixrowsep{0.4cm}
\xymatrixcolsep{0.1cm}\xymatrix @C=0.1em {
   \scriptstyle{(3,1)} & &\scriptstyle{(3,2)}\ar[rd]  & &\scriptstyle{(3,3)}\\       
   &\scriptstyle{(2,1)}& &\scriptstyle{(2,2)}\ar[ru]  & & \\
    & &\scriptstyle{(1,1)}& & & \\
}
\end{tabular}
\hspace{-0.8cm}
\switchcolumn
\hspace{-1cm}
\begin{tabular}{c c c c c}
\xymatrixrowsep{0.5cm}
\xymatrixcolsep{0.1cm}\xymatrix @C=0.1em {
  & \scriptstyle{(3,1)}\ar@{<-}[rd] & &\scriptstyle{(3,2)}\ar[rd]\ar[rr]  & &\scriptstyle{(3,3)}\\    &     &\scriptstyle{(2,1)}& &\scriptstyle{(2,2)}\ar[ru]\ar[ld]   & & \\
  &    & &\scriptstyle{(1,1)}& & & \\
}
\end{tabular}
\end{paracol}

Thanks to Theorem~\ref{theo:main2}, we can describe a basis for any simple subquotient of $V_{G}(T(L))$:
\begin{align*}
\mathcal{B}_1&:=\left\{T(L+z)\in \mathcal{B}_{G}(T(L))\, |\,
-z_{21}-2\in\mathbb{Z}_{\ge 0}\text{ and  } z_{22}-z_{11}\in\mathbb{Z}_{\ge 0}
\right\},\\
\mathcal{B}_2&:=\left\{T(L+z)\in \mathcal{B}_{G}(T(L))\, |\,
z_{21}+2\in\mathbb{Z}_{> 0}\text{ and  }z_{22}-z_{11}\in\mathbb{Z}_{\geq 0}
\right\},\\
\mathcal{B}_3&:=\left\{T(L+z)\in \mathcal{B}_{G}(T(L))\, |\,
-z_{21}-2\in\mathbb{Z}_{\geq 0}\text{ and  }z_{11}-z_{22}\in\mathbb{Z}_{> 0}
\right\},\\
\mathcal{B}_4&:=\left\{T(L+z)\in \mathcal{B}_{G}(T(L))\, |\,
z_{21}+2\in\mathbb{Z}_{> 0}\text{ and  }z_{11}-z_{22}\in\mathbb{Z}_{> 0}
\right\}.
\end{align*}
In addition, $\mathcal{B}_1$ is the unique simple submodule of $V_G(T(L))$, and by Corollary \ref{cor:geradoigualrelation}, any tableau in $\mathcal{B}_4$ generates the module $V_G(T(L))$.
\end{example}

\begin{example}
Let us consider $G$ and $T(R)$ from Example~\ref{ex:diamantecomp}.
Theorem \ref{theo:main2} allows us to describe the subquotients of $V_G(T(L))$ in terms of $G(Q)$ for all $T(Q)\in\mathcal{B}_G(T(L))$. In the table below, we omit rows $1$ and $2$ related to the arrows in the subgraph $\overline{G}$ of each $G(Q)$. Each of the following graphs corresponds to a subquotient of the module $V_G(T(R))$:

\begin{center}
\begin{tabular}{|c|c|}
\hline
{\tiny $\begin{array}{ccc}
& & \\ \\&\text{\begin{tabular}{c c}
\xymatrixrowsep{0.5cm}
\xymatrixcolsep{0.1cm}\xymatrix @C=0.1em {
 \scriptstyle{(4,1)}& &\scriptstyle{(4,2)}\ar[rr]    & &\scriptstyle{(4,3)} &   &\scriptstyle{(4,4)} \\
  & \scriptstyle{(3,1)}\ar[lu]    & &\scriptstyle{(3,2)}\ar[lu]\ar[ru]  & &\scriptstyle{(3,3)}\ar[ru]
}
\end{tabular}}
&
\end{array}$}  
\hspace{-0.9cm}
&
\hspace{-0.9cm}
{\tiny $\begin{array}{ccc}
& & \\ \\&\text{\begin{tabular}{c c}
\xymatrixrowsep{0.5cm}
\xymatrixcolsep{0.1cm}\xymatrix @C=0.1em {
 \scriptstyle{(4,1)}& &\scriptstyle{(4,2)}\ar[rr]    & &\scriptstyle{(4,3)} &   &\scriptstyle{(4,4)} \\
  & \scriptstyle{(3,1)}\ar[lu]    & &\scriptstyle{(3,2)}\ar@{<-}[lu]\ar[ru]  & &\scriptstyle{(3,3)}\ar[ru]
}
\end{tabular}}&
\end{array}$}
   \\
\hline
\end{tabular}
\begin{tabular}{|c|c|}
\hline

{\tiny $\begin{array}{ccc}
& & \\ \\&\text{\begin{tabular}{c c}
\xymatrixrowsep{0.5cm}
\xymatrixcolsep{0.1cm}\xymatrix @C=0.1em {
 \scriptstyle{(4,1)}& &\scriptstyle{(4,2)}\ar[rr]    & &\scriptstyle{(4,3)} &   &\scriptstyle{(4,4)} \\
  & \scriptstyle{(3,1)}\ar[lu]    & &\scriptstyle{(3,2)}\ar@{<-}[lu]\ar@{<-}[ru]  & &\scriptstyle{(3,3)}\ar[ru]\\  
}
\end{tabular}}&
\end{array}$}  
\hspace{-0.9cm}
&
\hspace{-0.9cm}
{\tiny $\begin{array}{ccc}
& & \\ \\&\text{\begin{tabular}{c c}
\xymatrixrowsep{0.5cm}
\xymatrixcolsep{0.1cm}\xymatrix @C=0.1em {
 \scriptstyle{(4,1)}& &\scriptstyle{(4,2)}\ar[rr]    & &\scriptstyle{(4,3)} &   &\scriptstyle{(4,4)} \\
  & \scriptstyle{(3,1)}\ar@{<-}[lu]    & &\scriptstyle{(3,2)}\ar[lu]\ar[ru]  & &\scriptstyle{(3,3)}\ar[ru]\\  
}
\end{tabular}}&
\end{array}$} \\
\hline
\end{tabular}
\begin{tabular}{|c|c|}
\hline

{\tiny $\begin{array}{ccc}
& & \\ \\&\text{\begin{tabular}{c c}
\xymatrixrowsep{0.5cm}
\xymatrixcolsep{0.1cm}\xymatrix @C=0.1em {
 \scriptstyle{(4,1)}& &\scriptstyle{(4,2)}\ar[rr]    & &\scriptstyle{(4,3)} &   &\scriptstyle{(4,4)} \\
  & \scriptstyle{(3,1)}\ar@{<-}[lu]    & &\scriptstyle{(3,2)}\ar@{<-}[lu]\ar[ru]  & &\scriptstyle{(3,3)}\ar[ru]\\  
}
\end{tabular}}&
\end{array}$}
\hspace{-0.9cm}
&
\hspace{-0.9cm}
{\tiny $\begin{array}{ccc}
& & \\ \\&\text{\begin{tabular}{c c}
\xymatrixrowsep{0.5cm}
\xymatrixcolsep{0.1cm}\xymatrix @C=0.1em {
 \scriptstyle{(4,1)}& &\scriptstyle{(4,2)}\ar[rr]    & &\scriptstyle{(4,3)} &   &\scriptstyle{(4,4)} \\
  & \scriptstyle{(3,1)}\ar@{<-}[lu]    & &\scriptstyle{(3,2)}\ar@{<-}[lu]\ar@{<-}[ru]  & &\scriptstyle{(3,3)}\ar[ru]\\  
}
\end{tabular}}&
\end{array}$} \\
\hline
\end{tabular}
\begin{tabular}{|c|c|}
\hline

{\tiny $\begin{array}{ccc}
& & \\ \\&\text{\begin{tabular}{c c}
\xymatrixrowsep{0.5cm}
\xymatrixcolsep{0.1cm}\xymatrix @C=0.1em {
 \scriptstyle{(4,1)}& &\scriptstyle{(4,2)}\ar[rr]    & &\scriptstyle{(4,3)} &   &\scriptstyle{(4,4)} \\
  & \scriptstyle{(3,1)}\ar[lu]    & &\scriptstyle{(3,2)}\ar[lu]\ar[ru]  & &\scriptstyle{(3,3)}\ar@{<-}[ru]\\  
}
\end{tabular}}&
\end{array}$}  
\hspace{-0.9cm}
&
\hspace{-0.9cm}
{\tiny $\begin{array}{ccc}
& & \\ \\&\text{\begin{tabular}{c c}
\xymatrixrowsep{0.5cm}
\xymatrixcolsep{0.1cm}\xymatrix @C=0.1em {
 \scriptstyle{(4,1)}& &\scriptstyle{(4,2)}\ar[rr]    & &\scriptstyle{(4,3)} &   &\scriptstyle{(4,4)} \\
  & \scriptstyle{(3,1)}\ar[lu]    & &\scriptstyle{(3,2)}\ar@{<-}[lu]\ar[ru]  & &\scriptstyle{(3,3)}\ar@{<-}[ru]\\  
}
\end{tabular}}&
\end{array}$}
   \\
\hline
\end{tabular}
\begin{tabular}{|c|c|}
\hline

{\tiny $\begin{array}{ccc}
& & \\ \\&\text{\begin{tabular}{c c}
\xymatrixrowsep{0.5cm}
\xymatrixcolsep{0.1cm}\xymatrix @C=0.1em {
 \scriptstyle{(4,1)}& &\scriptstyle{(4,2)}\ar[rr]    & &\scriptstyle{(4,3)} &   &\scriptstyle{(4,4)} \\
  & \scriptstyle{(3,1)}\ar[lu]    & &\scriptstyle{(3,2)}\ar@{<-}[lu]\ar@{<-}[ru]  & &\scriptstyle{(3,3)}\ar@{<-}[ru]\\  
}
\end{tabular}}&
\end{array}$}  
\hspace{-0.9cm}
&
\hspace{-0.9cm}
{\tiny $\begin{array}{ccc}
& & \\ \\&\text{\begin{tabular}{c c}
\xymatrixrowsep{0.5cm}
\xymatrixcolsep{0.1cm}\xymatrix @C=0.1em {
 \scriptstyle{(4,1)}& &\scriptstyle{(4,2)}\ar[rr]    & &\scriptstyle{(4,3)} &   &\scriptstyle{(4,4)} \\
  & \scriptstyle{(3,1)}\ar@{<-}[lu]    & &\scriptstyle{(3,2)}\ar[lu]\ar[ru]  & &\scriptstyle{(3,3)}\ar@{<-}[ru]\\  
}
\end{tabular}}&
\end{array}$} \\
\hline
\end{tabular}
\begin{tabular}{|c|c|}
\hline

{\tiny $\begin{array}{ccc}
& & \\ \\&\text{\begin{tabular}{c c}
\xymatrixrowsep{0.5cm}
\xymatrixcolsep{0.1cm}\xymatrix @C=0.1em {
 \scriptstyle{(4,1)}& &\scriptstyle{(4,2)}\ar[rr]    & &\scriptstyle{(4,3)} &   &\scriptstyle{(4,4)} \\
  & \scriptstyle{(3,1)}\ar@{<-}[lu]    & &\scriptstyle{(3,2)}\ar@{<-}[lu]\ar[ru]  & &\scriptstyle{(3,3)}\ar@{<-}[ru]\\  
}
\end{tabular}}&
\end{array}$}
\hspace{-0.9cm}
&
\hspace{-0.9cm}
{\tiny $\begin{array}{ccc}
& & \\ \\&\text{\begin{tabular}{c c}
\xymatrixrowsep{0.5cm}
\xymatrixcolsep{0.1cm}\xymatrix @C=0.1em {
 \scriptstyle{(4,1)}& &\scriptstyle{(4,2)}\ar[rr]    & &\scriptstyle{(4,3)} &   &\scriptstyle{(4,4)} \\
  & \scriptstyle{(3,1)}\ar@{<-}[lu]    & &\scriptstyle{(3,2)}\ar@{<-}[lu]\ar@{<-}[ru]  & &\scriptstyle{(3,3)}\ar@{<-}[ru]\\  
}
\end{tabular}}&
\end{array}$} \\
\hline
\end{tabular}
\end{center}
\end{example}

\end{document}